\documentclass[11pt,a4paper]{article}
\title{\bf Semi-log-convexity of ${\rm M}/{\rm M}/\infty$ queues on $\mathbb{Z}_+$}
\author{Huige Chen\footnote{Email: {\color{blue}chenhuige@baicgroup.com.cn}.}
\quad Huaiqian Li\footnote{Email: {\color{blue}huaiqian.li@tju.edu.cn}.  Partially supported by the National Key R\&D Program of China (No. 2022YFA1006000).}
\vspace{2mm}
\\
{\footnotesize $^{*)}$Beijing Automotive Research Institute Co., Ltd, Beijing 100176, China}
\\
{\footnotesize $^{\dag)}$Center for Applied Mathematics and KL-AAGDM, Tianjin University, Tianjin 300072, China}
}

\date{}

\usepackage{amssymb,amsmath,amsfonts,amsthm,mathrsfs}
\usepackage{bbm}
\usepackage[marginal]{footmisc}
\usepackage{color}
\usepackage[]{hyperref}
\def\R{\mathbb{R}}
\def\E{\mathbb{E}}

\def\Z{\mathbb{Z}}

\def\M{\mathcal{M}}
\def\d{\textup{d}}

\def\<{\langle}
\def\>{\rangle}
\def\Proof.{\noindent{\bf Proof. }}

\def\newdot{{\kern.8pt\cdot\kern.8pt}}

\newtheorem{theorem}{Theorem}[section]
\newtheorem{lemma}[theorem]{Lemma}

\theoremstyle{definition}\newtheorem{remark}[theorem]{Remark}

\begin{document}
\allowdisplaybreaks
\maketitle
\makeatletter 
\renewcommand\theequation{\thesection.\arabic{equation}}
\@addtoreset{equation}{section}
\makeatother 

\begin{abstract}
We solve the problem left in the recent paper by N. Gozlan et al [Potential Analysis 58, 2023, 123--158], establishing the semi-log-convexity of semigroups associated with ${\rm M}/{\rm M}/\infty$ queuing processes on the set of non-negative integers. Our approach is global in nature and yields the sharp constant.
\end{abstract}



\section{Introduction and main results}\hskip\parindent
Let $\mathbb{Z}_+$ denote the set of non-negative integers. In this note, we are concerned with ${\rm M}/{\rm M}/\infty$ queuing processes taking values in  $\mathbb{Z}_+$; for further background, see e.g. \cite{GLMRS,Chafai,Asmussen}.

Let $\mathbb{N}=\mathbb{Z}_+\setminus\{0\}$ and let $\lambda,\mu>0$. We consider the ${\rm M}/{\rm M}/\infty$ queuing processes $(X_t)_{t\geq0}$ on $\mathbb{Z}_+$ with input (or arrival) rate $\lambda$ and service rate $\mu$. In other words, $(X_t)_{t\geq0}$ is a continuous-time Markov chain with the infinitesimal generator $\M$ given by
\begin{equation*}\label{}
    \mathcal{M}f(n):=\lambda [f(n+1)-f(n)] + n \mu[f(n-1)-f(n)],\quad n \in \mathbb{Z}_+,
\end{equation*}
for any function $f:\Z_+\to\R$, where $f(-1)$ is identified as $f(0)$. Let $(A_t)_{t\geq0}$ be the semigroup generated by $\M$. For convenience, we define the traffic intensity $\rho=\frac{\lambda}{\mu}$ and the function  $p=p_t=e^{-\mu t}$ for every $t\geq0$. It is well known that the stationary distribution of $(X_t)_{t\geq0}$ is the Poisson law with parameter $\rho$, denoted by $\pi_\rho$, where
\begin{equation}
\nonumber
    \pi_\rho (k) = \frac{\rho^k}{k!}e^{-\rho },\quad k \in \Z_+.
\end{equation}

 Recall that the discrete Laplacian $\Delta_{\rm d}$, acting on function $f:\Z_+\to\R$, is defined as
\begin{equation*}
    \Delta_{\rm d} f(n)= f(n+1)+f(n-1)-2 f(n), \quad n\in \mathbb{N}.
\end{equation*}
In a recent work \cite[Proposition 3.1]{GLMRS}, N. Gozlan et al proved a semi-log-convexity property for the semigroup $A_t$. Specifically, they proved that for every $t>0$ and every non-zero function $f: \Z_+\to[0,\infty)$, the following inequality holds:
\begin{equation}\label{GLMRS-1}
    \Delta_{\rm d} [\log A_{t} f](n) \ge \log \left[\frac{1}{12}\left(1-\frac{p^{2}}{\left[p+\rho (1-p)^{2}\right]^{2}}\right)\right],\quad n\in\mathbb{N}.
\end{equation}

It is easy to see that \eqref{GLMRS-1} is not sharp. Indeed, as $t\rightarrow\infty$, the right-hand side of \eqref{GLMRS-1} tends to $-\log( 12)$, while the left-hand side may vanish. This follows from the fact that for any $n\in\Z_+$ and any bounded $f: \Z_+\to[0,\infty)$, $A_tf(n)\stackrel{t\rightarrow\infty}{\to }\int_{\Z_+}f\,\d\pi_{\rho}$.  Due to this discrepancy, the problem of improving \eqref{GLMRS-1} remains open, as noted in \cite[Remark 3.2]{GLMRS}. Recall that a function $h:\Z_+\rightarrow[0,\infty)$ is called log-convex if for any $n\in\mathbb{N}$, $\Delta_{\rm d}h(n)\geq0$, or equivalently, $h(n)^2\leq h(n+1)h(n-1)$. Inequality \eqref{GLMRS-1} provides a lower bound for $\Delta_{\rm d} (\log A_{t} f)$, which is weaker than full log-convexity but retains a similar flavor\textemdash{}hence the term semi-log-convexity for $A_t$.

The inequality \eqref{GLMRS-1}, including inequalities of this type in a broad sense, are of significant interest for several reasons. On the one hand, they play a crucial role in the approach developed in \cite{EldanLee} to study Talagrand's convolution conjecture (also known as the $L^1$-regularization effect) for the Ornstein--Uhlenbeck semigroup on $\R^d$; see \cite{Lehec} for the complete resolution of this case. The original conjecture, formulated by Talagrand \cite{Tala} for the Hamming (or Boolean) cube, remains open, while recent work in \cite{GLMRS} has extended these investigations to other models including the ${\rm M}/{\rm M}/\infty$ queue. On the other hand, inequality \eqref{GLMRS-1}  exhibits deep connections to non-local Li--Yau type inequalities investigated recently. For instance, in \cite{WZ}, such inequalities are established for the fractional Laplacian on $\R^d$, with a related discrete model discussed in Section 4 of the same paper,  and in \cite{LiQian}, analogous results are derived for a class of non-local Schr\"{o}dinger operators on $\R^d$, namely, Dunkl harmonic oscillators.

Our main contribution, contained in the following theorem, improves \eqref{GLMRS-1} by eliminating the extra constant $\frac{1}{12}$.
\begin{theorem}\label{main}
For every non-zero function $f: \Z_+\to[0,\infty)$ and every $t>0$,
        \begin{equation}\label{main-1}
    \Delta_{\rm d} [\log A_{t} f](n) \ge \log \left(1-\frac{p^{2}}{\left[p+\rho (1-p)^{2}\right]^{2}}\right),\quad n\in\mathbb{N}.
\end{equation}
 \end{theorem}
\begin{remark}
Theorem \ref{main} is sharp in the following sense: for every bounded function $f: \Z_+\to[0,\infty)$, both sides of \eqref{main-1} go to $0$ as $t\rightarrow\infty$.
\end{remark}

In the remaining part of this note, we aim to prove our main results.

\section{Proofs of Theorem \ref{main}}\hskip\parindent
We begin with some preliminary definitions.  Let $\mathcal{B}(k,a)$ be the binomial law of parameters $k\in\Z_+$ and $a\in[0,1]$, adopting the conventions that $\mathcal{B}(k,0)=\delta_0$ and $\mathcal{B}(k,1)=\delta_k$, where $\delta_n$ stands for the Dirac measure at $n$. Recall that
 $$\mathcal{B}(k,a)=\sum_{j=0}^k\frac{k!}{j!(k-j)!}a^j(1-a)^{k-j}\delta_j.$$

It is well known that, by the discrete analog of the Mehler formula (see equation (3) on page 321 of \cite{Chafai}), the law of the ${\rm M}/{\rm M}/\infty$ queue $X_t$ introduced above can be represented as the convolution of the  binomial law and the Poisson law. More precisely, for every $t\geq0$ and every $k\in \Z_+$,
\begin{equation}\label{mehler-1}
    {\rm Law}(X_{t} | X_{0}=k)=\mathcal{B}(k, p_t) \ast \pi_{\rho q_t},
\end{equation}
where $\ast$ denotes the discrete convolution, $q_t=1-p_t$ and  recall that $\rho=\frac{\lambda}{\mu}$ and $p=p_t=e^{-\mu t}$ for every $t\geq0$. This decomposition corresponds to the sum $X_t=Y_t+Z_t$, where $(Y_t)_{t\geq0}$ and $(Z_t)_{t\geq0}$ are independent processes such that ${\rm Law}(Y_t)=\mathcal{B}(k,p_t)$ and ${\rm Law}(Z_t)=\pi_{\rho q_t}$. Consequently, the semigroup action admits the representation:
\begin{eqnarray}\label{mehler-2}
A_t f(k):=\E[f(X_t)|X_0=k]=\E[f(Y_t+Z_t)],\quad k\in\Z_+,
\end{eqnarray}
for every function $f:\Z_+\rightarrow[0,\infty)$. Moreover, the process $(X_t)_{t\geq0}$ is reversible with respect to the Poisson law $\pi_\rho$:
\begin{equation}\label{rever}
    \mathbb{P}(X_{t}=i \mid X_{0}=j)\pi_{\rho}(j)= \mathbb{P}(X_{t}=j \mid X_{0}=i)\pi_{\rho}(i),\quad i, j \in \Z_+.
\end{equation}

To prove our main theorem, the next lemma is crucial.
\begin{lemma}\label{keylemma} Let $t>0$. Denote
\begin{equation}\label{notation-G}
G_{k}(n)=\mathbb{P}\left(X_{t}=n \mid X_{0}=k\right) ,\quad  k, n \in \mathbb{Z}_+,
\end{equation}
    and set
   \begin{equation}
   \nonumber
    K^{-1}=\frac{n}{n+1}\left(1-\frac{p^{2}}{[\rho(1-p)^{2}+p]^{2}}\right),\quad n \in \mathbb{N}.
\end{equation}
Then, for every $ k\in \mathbb{Z}_+$ and every $n \in \mathbb{N}$,
\begin{equation}\label{keylemma-1}
    G_{k}(n)^{2} \le K G_{k}(n+1) G_{k}(n-1).
\end{equation}
\end{lemma}
\begin{remark}\label{remark-keylemma}
(1) Recall that a function $h:\Z_+\rightarrow[0,\infty)$ is called ultra-log-convex if the mapping $\Z_+\ni n\mapsto h(n)n!$ is log-convex, or equivalently, $h(n)^2\leq \frac{n+1}{n}h(n+1)h(n-1)$ for all $n\in\mathbb{N}$. According to this, since $K\in(\frac{n+1}{n},\infty)$ clearly, we may interpret inequality \eqref{keylemma-1} as the semi-ultra-log-convexity of the function $\mathbb{N}\ni n\mapsto G_k(n)$ for each fixed $k\in\Z_+$. For further background on the log-convexity and the ultra-log-convexity, including their various properties, applications and relationships, we recommend the interested reader to the comprehensive review \cite{SW2014}.

(2) Lemma \ref{keylemma} can be generalized as follows. Let $\widetilde{Y}$ and $\widetilde{Z}$ be independent random variables such that ${\rm Law}(\widetilde{Y})=\mathcal{B}(k,a)$ and ${\rm Law}(\widetilde{Z})=b$, where $a,b\in[0,1]$ and $k\in\mathbb{Z}_+$. Define
$$H_k(n)=[\mathcal{B}(k,a)\ast\pi_b](n),\quad k,n\in\mathbb{Z}_+,$$
and set
$$M^{-1}=\frac{n}{n+1}\left(1-\frac{a^{2}}{[(1-a)b+a]^{2}}\right),\quad n \in \mathbb{N}.$$
Then
\begin{equation}\label{remark-keylemma-1}
    H_{k}(n)^{2} \le M H_{k}(n+1) H_{k}(n-1),\quad k\in\mathbb{Z}_+,\,n\in\mathbb{N}.
\end{equation}
The proof of \eqref{remark-keylemma-1} follows the same method as that used for \eqref{keylemma-1}. Furthermore, for a fixed $t>0$, if we take ${\rm Law}(Y_t)={\rm Law}(\widetilde{Y})$, ${\rm Law}(Z_t)={\rm Law}(\widetilde{Z})$, $a=p$ and $b=\rho(1-p)$, then \eqref{remark-keylemma-1} reduces to \eqref{keylemma-1}.
\end{remark}

\begin{proof}[Proof of Lemma \ref{keylemma}] Fix $t>0$. Let $Y_t$ and $Z_t$ be independent such that ${\rm Law}(Y_t)=\mathcal{B}(k, p)$ and ${\rm Law}(Z_t)=\pi_{\rho(1-p)}$. Then, by \eqref{mehler-2}, we clearly have
\begin{eqnarray*}
G_{k}(n) 
&=&\mathbb{P}\left(Y_{t}+Z_{t}=n\right),\quad  k,n \in \mathbb{Z}_+.
\end{eqnarray*}

Now we prove \eqref{keylemma-1} by induction.

(1) Let $k=0$. Then, since $\mathbb{P}(Y_{t}=0) = 1$ and ${\rm Law}(Z_{t})=\pi_{\rho(1-p)}$, by the independence, we have
\begin{eqnarray*}
G_{0}(n)&=&\mathbb{P}(Y_{t}+Z_{t}=n)=\mathbb{P}(Y_{t}=0) \mathbb{P}(Z_{t}=n)\\
 &=& \mathbb{P}(Z_{t}=n) = \frac{[\rho(1-p)]^{n}}{n !}  e^{-\rho(1-p)},\quad n \in \mathbb{N}.
\end{eqnarray*}
Hence
\begin{eqnarray*}
&&\frac{G_{0}(n+1) G_{0}(n-1)}{G_{0}(n)^{2}} =\frac{n}{n+1} \\
&&\ge \frac{n}{n+1}\left(1-\frac{p^{2}}{[\rho(1-p)^{2}+p]^{2}}\right)
=K^{-1}, \quad n \in \mathbb{N}.
\end{eqnarray*}

(2) Let $k=1$. Then, since ${\rm Law}(Y_t)= \mathcal{B}(1, p)$ and ${\rm Law}( Z_{t} )= \pi_{\rho(1-p)}$, by the independence, we obtain
\begin{equation*}\begin{split}
\nonumber
G_{1}(n) & =\mathbb{P}(Y_{t}=1) \mathbb{P}(Z_{t}=n-1)+\mathbb{P}(Y_{t}=0) \mathbb{P}(Z_{t}=n) \\
& =p \pi_{\rho(1-p)}(n-1)+(1-p) \pi_{\rho(1-p)}(n) \\
& =\left((1-p)+p \frac{n}{\rho(1-p)}\right) \frac{[\rho(1-p)]^{n}}{n !}  e^{-\rho(1-p)},\quad n\in \mathbb{N}.
\end{split}\end{equation*}
Hence
\begin{equation*}
\begin{split}
\nonumber
    \frac{G_{1}(n+1) G_{1}(n-1)}{G_{1}(n)^{2}}&=\frac{n}{n+1}  \frac{\left((1-p)+p \frac{n+1}{\rho(1-p)}\right)\left((1-p)+p \frac{n-1}{\rho(1-p)}\right)}{\left((1-p)+p \frac{n}{\rho(1-p)}\right)^{2}}\\
    & =\frac{n}{n+1} \frac{\left((1-p)+p \frac{n}{\rho(1-p)}\right)^{2}-\left(\frac{p}{\rho(1-p)}\right)^{2}}{\left((1-p)+p \frac{n}{\rho(1-p)}\right)^{2}} \\
    & \ge \frac{n}{n+1}\left(1-\frac{p^{2}}{[\rho(1-p)^{2}+p]^{2}}\right)\\
    & =K^{-1}, \quad n \in \mathbb{N}.
    \end{split}
\end{equation*}

(3) Let $k=m\in\mathbb{N}\setminus\{1\}$. Suppose that for any $m>2$,
\begin{equation}\label{3-1}
    G_{m-1}(n)^{2} \leq K G_{m-1}(n+1) G_{m-1}(n-1) ,\quad n \in \mathbb{N}.
\end{equation}
 We need to prove that
\begin{equation}\label{3-2}
    G_{m}(n)^{2} \le K G_{m}(n+1) G_{m}(n-1), \quad n \in \mathbb{N}.
\end{equation}
Below, we divide the proof of \eqref{3-2} into two parts according to that $n=1$ and $n\in\mathbb{N}\setminus\{1\}$.

(i) Let $n=1$. It suffices to prove that
\begin{equation}\label{3(i)-1}
G_{m}(1)^2 \leq K G_{m}(2) G_{m}(0).
\end{equation}
For convenience, let $b = \rho(1-p)$. By the independence, we have
\begin{equation}
\begin{split}
\nonumber
G_{m}(0) &=\mathbb{P}\left(Y_{t}+Z_{t}=0\right) =\mathbb{P}\left(Y_t=0\right) \mathbb{P}\left(Z_{t}=0\right) \\
& =(1-p)^{m} e^{-b},\\
G_{m}(1) & =\mathbb{P}\left(Y_{t}+Z_ t=1\right) \\
& =\mathbb{P}\left(Y_{t}=1\right) \mathbb{P}\left(Z_{t}=0\right)+\mathbb{P}\left(Y_{t}=0\right) \mathbb{P}\left(Z_{t}=1\right) \\
& =m p(1-p)^{m-1}e^{-b}+(1-p)^{m} b e^{-b}, \\
G_{m}(2) & =\mathbb{P}\left(Y_{t}+Z_{t}=2\right) \\
& =\mathbb{P}\left(Y_{0}=0\right) \mathbb{P}\left(Z_{t}=2\right)+\mathbb{P}\left(Y_t=1\right) \mathbb{P}\left(Z_t=1\right)+\mathbb{P}\left(Y_t=2\right) \mathbb{P}\left(Z_{t}=0\right) \\
& =(1-p)^{m} \frac{b^{2}}{2} e^{-b}+m p(1-p)^{m-1} b e^{-b}+\frac{m(m-1)}{2} p^{2}(1-p)^{m-2} e^{-b}.
\end{split}
\end{equation}
Hence, to show that \eqref{3(i)-1} holds, it is equivalent to prove that
\begin{equation}\label{3(i)-2}
[m p+(1-p) b]^{2} \leq K \left[ \frac{(1-p)^{2} b^{2}}{2}+m p(1-p) b+\frac{m(m-1)}{2} p^{2}\right],
\end{equation}
where
$$
K^{-1}=\frac{1}{2}\left(1-\frac{p^{2}}{[b(1-p)+p ]^{2}}\right).
$$
Let $u = b(1-p)$. Then, \eqref{3(i)-2} can be rewritten as
$$
(m p+u)^{2}\left(u^{2}+2 p u\right) \leq (u+p)^{2}\left[u^{2}+2 m p u+m(m-1) p^{2}\right],
$$
which is clearly equivalent to
\begin{equation}\label{3(i)-3}
    (1-m)(p^2u^2-mp^4) \ge 0.
\end{equation}
By \eqref{3-1}, $G_{m-1}(1)^{2} \leq K G_{m-1}(2) G_{m-1}(0)$ for any $m\in\mathbb{N}$ such that $m>2$, which means that
$$
(2-m)[p^2u^2-(m-1)p^4] \ge 0.
$$
Hence
$$
p^2u^2 \le (m-1)p^4\le mp^4.
$$
Thus, \eqref{3(i)-3} holds for any $m\in\mathbb{N}\setminus\{1\}$, from which we conclude that \eqref{3(i)-1} holds.

(ii) Let $n \in \mathbb{N}\setminus\{1\}$. Let $Y_{t}=Y_{t}'+\varepsilon_{t}$, where ${\rm Law}(Y_{t}')= \mathcal{B}(m-1, p)$ and $ {\rm Law}(\varepsilon_{t})= \mathcal{B}(1, p)$ such that $Y_{t}'$ and $\varepsilon_{t}$ are independent and also independent of $Z_{t}$ with ${\rm Law}(Z_{t})= \pi_{\rho(1-p)}$. Then
\begin{equation}
\begin{split}
\nonumber
    G_{m}(n)&=\mathbb{P}(Y_{t}+Z_{t}=n) =\mathbb{P}(Y_{t}'+\varepsilon_{t}+Z_{t}=n) \\
    & =\sum_{j=0}^{n} \mathbb{P}(Y_{t}'+Z_{t}=j) \mathbb{P}(\varepsilon_{t}=n-j)\\
    & = \sum_{j=1}^{n} \mathbb{P}(Y_{t}{ }^{\prime}+Z_{t}=j) \mathbb{P}(\varepsilon_{t}=n-j),
\end{split}
\end{equation}
where the last equality is due to the fact that $\mathbb{P}(\varepsilon_{t} = n)=0$ since $n>1$. Thus, by the Cauchy--Schwarz inequality, we derive that
\begin{equation}
\begin{split}
\nonumber
    &G_{m}(n) = \sum_{j=1}^{n} \mathbb{P}(Y_{t}'+Z_{t}=j) \mathbb{P}(\varepsilon_{t}=n-j)\\
    &\le \sum_{j=1}^{n} \sqrt{K}\, \mathbb{P}(Y_{t}^{\prime}+Z_{t}=j+1)^{\frac{1}{2}} \mathbb{P}(Y_{t}^{\prime}+Z_{t}=j-1)^{\frac{1}{2}}\,
          \mathbb{P}(\varepsilon_{t}=n-j)\\
    &\leq \sqrt{K}\,\bigg(\sum_{j=1}^{n} \mathbb{P}(Y_{t}'+Z_{t}=j+1) \mathbb{P}(\varepsilon_{t}=n-j)\bigg)^{\frac{1}{2}}
          \bigg(\sum_{j=1}^{n} \mathbb{P}(Y_{t}'+Z_{t}=j-1) \mathbb{P}(\varepsilon_{t}=n-j)\bigg)^{\frac{1}{2}}\\
    & = \sqrt{K}\bigg(\sum_{j=2}^{n+1} \mathbb{P}(Y_{t}'+Z_{t}=j) \mathbb{P}(\varepsilon_{t}=n-j+1)\bigg)^{\frac{1}{2}}\bigg(\sum_{j=0}^{n-1}
          \mathbb{P}(Y_{t}'+Z_{t}=j) \mathbb{P}(\varepsilon_{t}=n-j-1)\bigg)^{\frac{1}{2}} \\
    &\le \sqrt{K}\bigg(\sum_{j=0}^{n+1} \mathbb{P}(Y_{t}'+Z_{t}=j) \mathbb{P}(\varepsilon_{t}=n-j+1)\bigg)^{\frac{1}{2}}\bigg(\sum_{j=0}^{n-1}
          \mathbb{P}(Y_{t}'+Z_{t}=j) \mathbb{P}(\varepsilon_{t}=n-j-1)\bigg)^{\frac{1}{2}} \\
    & =\sqrt{K}\,\mathbb{P}(Y_{t}'+Z_{t}+\varepsilon_{t}=n+1)^{\frac{1}{2}} \mathbb{P}(Y_{t}'+Z_{t}+\varepsilon_{t}=n-1)^{\frac{1}{2}} \\
    & =\sqrt{K}\, G_{m}(n+1)^{\frac{1}{2}} G_{m}(n-1)^{\frac{1}{2}}.
\end{split}
\end{equation}

Putting (i) and (ii) together, we prove \eqref{3-2}.

Therefore, combing (1), (2) and (3) together, we complete the proof of \eqref{keylemma-1}.
\end{proof}

Now we are ready to prove our main result.
\begin{proof}[Proof of Theorem \ref{main}]
Let $t>0$ and let $f: \mathbb{N}\to[0,\infty)$ be not identical to $0$. By the proof of \cite[Proposition 3.5]{GLMRS}, we have
\begin{equation*}
    \Delta_{\rm d} [\log A_{t} f](n)=\log \frac{A_{t} f(n+1) A_{t} f(n-1)}{A_{t} f(n)^{2}},\quad n \in \mathbb{N}.
\end{equation*}
Then, combining this with  \eqref{mehler-2} and \eqref{rever}, we derive that
\begin{eqnarray}\label{proof-main-4}
    &&\Delta_{\rm d} [\log A_{t} f](n)\cr
    &&=\log \frac{\big(\sum_{k=0}^{\infty} f(k) \mathbb{P}(X_{t}=k \mid X_{0}=n+1)\big)\big(\sum_{k=0}^{\infty} f(k) \mathbb{P}(X_{t}=k \mid X_{0}=n-1)\big)}{\big(\sum_{k=0}^{\infty} f(k) \mathbb{P}(X_{t}=k \mid X_{0}=n)\big)^{2}} \cr
    &&=\log \left[\frac{n+1}{n} \cdot \frac{\big(\sum_{k=0}^{\infty} f(k) \pi_{\rho}(k) G_{k}(n+1)\big)\big(\sum_{k=0}^{\infty} f(k) \pi_{\rho}(k) G_{k}(n-1)\big)}{\big(\sum_{k=0}^{\infty} f(k) \pi_{\rho}(k) G_{k}(n)\big)^{2}}\right]\cr
    &&\ge \log \left[\frac{n+1}{n} \cdot \frac{\left(\sum_{k=0}^{\infty} f(k) \pi_{\rho}(k) \sqrt{G_{k}(n+1)} \sqrt{G_{k}(n-1)}\right)^{2}}{\big(\sum_{k=0}^{\infty} f(k) \pi_{\rho}(k) G_{k}(n)\big)^{2}}\right],\quad n \in \mathbb{N},
\end{eqnarray}
where $G_k(n)$ is defined in \eqref{notation-G} and the Cauchy--Schwarz inequality is applied in the last inequality. Thus, by \eqref{keylemma-1} and \eqref{proof-main-4}, we immediately have
$$
\Delta_{\rm d} [\log A_{t} f](n) \ge \log \left(\frac{n+1}{n}  K^{-1}\right)=\log \left(1-\frac{p^{2}}{[\rho(1-p)^{2}+p]^{2}}\right),\quad n \in \mathbb{N}.
$$

The proof is completed.
\end{proof}

\subsection*{Acknowledgment}\hskip\parindent
The authors sincerely thank the anonymous referee for his/her insightful comments and valuable suggestions, which have significantly improved the manuscript.

\end{document}